\documentclass[10pt]{amsart}

\usepackage{amsmath, amsthm, amscd, amsfonts, amssymb, graphicx, color}
\usepackage{latexsym}
\usepackage{amscd}
\usepackage{amsthm}
\usepackage{amssymb}
\usepackage{enumerate}
\usepackage{mathabx}

\theoremstyle{plain}
\newtheorem{theorem}{Theorem}
\newtheorem{corollary}[theorem]{Corollary}
\newtheorem{lemma}[theorem]{Lemma}
\newtheorem{prop}[theorem]{Proposition}
\theoremstyle{remark}
\theoremstyle{definition}
\newtheorem{remark}[theorem]{Remark}
\newtheorem{definition}[theorem]{Definition}
\newtheorem{example}[theorem]{Example}

\newcommand{\R}{\mathbb{R}}

\renewcommand*{\span}{\ensuremath{\mathrm{span\,}}}

\begin{document}

\title{Characterization of Polynomials as solutions of certain functional equations}
\author{J.~M.~Almira}



\address{Departamento de Matem\'{a}ticas, Universidad de Ja\'{e}n, E.P.S. Linares,  Campus Cient\'{\i}fico Tecnol\'{o}gico de Linares, 23700 Linares, Spain}
\email{jmalmira@ujaen.es}

\maketitle
\begin{abstract}
In \cite{AS_LCFE} the  functional equation
\[
\sum_{i=0}^mf_i(b_ix+c_iy)= \sum_{i=1}^na_i(y)v_i(x)
\]
with $x,y\in\mathbb{R}^d$ and $b_i,c_i\in\mathbf{GL}_d(\mathbb{C})$, was studied, both in the classical context of continuous complex valued functions and in the  framework of complex valued Schwartz distributions, where these equations were properly introduced in two different ways. The solution sets of these equations are, typically,  exponential polynomials and, in some particular cases, they reduce to ordinary polynomials.
In this paper we present several characterizations of ordinary polynomials as the solution sets of certain related functional equations. Some of  these equations are important because of their connection with the  Characterization Problem of distributions in Probability Theory.
\end{abstract}

\markboth{J.~M.~Almira}{Characterizations of Polynomials}

\section{Introduction}
The Levi-Civita functional equation, which  has the form
\begin{equation} \label{LC}
f(x+y)=\sum_{i=1}^na_i(y)v_i(x),
\end{equation}
where $f, a_k, v_k$ are complex valued functions defined on a semigroup $(\Gamma,+)$,  can be restated by claiming that 
$\tau_y(f)\in W$ for all $y\in \Gamma$, where $W=\mathbf{span}\{v_k\}_{k=1}^n$ is a finite dimensional space of functions defined on $\Gamma$ and $\tau_y(f)(x)=f(x+y)$. 

If $\Gamma=\mathbb{R}^d$ for some $d\geq 1$, the equation \eqref{LC} can be formulated also for distributions, since the translation operator 
$\tau_y$ can be extended, in a natural way, to the space $\mathcal{D}(\mathbb{R}^d)'$  of Schwartz complex valued distributions. Concretely, we can define 
\[
\tau_y(f)\{\phi\}=f\{\tau_{-y}(\phi)\}
\]
for all $y\in \mathbb{R}^d$ and all test function $\phi$. 
For these distributions we will also consider, in this paper, the dilation operator
\[
\sigma_b(f)\{\phi\} =\frac{1}{|\det(b)|}f\{\sigma_{b^{-1}}(\phi)\}, 
\] 
where  $b\in\mathbf{GL}_d(\mathbb{C})$  is any invertible matrix,   $\phi \in\mathcal{D}(\mathbb{R}^d)$   is any test function, and 
$\sigma_{b^{-1}}(\phi)(x)=\phi(b^{-1}x)$ for all $x\in\mathbb{R}^d$. 

If  $X_d$ denotes  either the set of continuous complex valued functions $C(\mathbb{R}^d)$ or the set of Schwartz complex valued distributions 
$\mathcal{D}(\mathbb{R}^d)'$, and  $f\in X_d$, then it is known that 
$\tau_y(f)\in W$ for all $y\in \Gamma$, where $W=\mathbf{span}\{v_k\}_{k=1}^n$ is a finite dimensional subspace of $X_d$, if and only if 
 $f$ is equal, in distributional sense, to a continuous exponential polynomial (also named quasi-polynomial). Indeed, if we set $M=\tau(f)=\mathbf{span}\{f(\cdot+y):y\in\mathbb{R}^d\}$, then $M\subseteq W$ is finite dimensional and translation invariant, so that  Anselone-Korevaar's Theorem \cite{anselone} implies that all its elements, including $f(x)$, are exponential polynomials. This was proved in 1913 by Levi-Civita \cite{LC}  for the case of ordinary continuous functions (see also \cite{leland}, \cite{Lo} for other proofs) Furthermore, if $\{w_k\}_{k=1}^N$ is a basis of the translation invariant space $M$, then every $f\in M$ satisfies the family of equations
\[
\tau_yf=\sum_{i=1}^Nb_i(y)w_i \text{  } (y\in\mathbb{R}^d).
\]
Thus, in the context of distributions, it makes sense to say that $f\in \mathcal{D}(\mathbb{R}^d)'$ satisfies the Levi-Civita functional equation if  there exist distributions $\{v_1,\cdots,v_m\}\subseteq \mathcal{D}(\mathbb{R}^d)'$ and ordinary functions $a_i:\mathbb{R}^d\to\mathbb{C}$ such that, for every $y\in\mathbb{R}^d$
\begin{equation}\label{LCdistributions1}
\tau_{y}(f)=\sum_{i=1}^ma_i(y)v_i.
\end{equation}
Indeed, we can assume that $\mathbf{span}\{v_1,\cdots,v_m\}$ is  translation invariant of dimension $m$. Then Anselone-Korevaar's theorem guarantees that $v_1,..,v_m$ and $f$ are all of them continuous exponential polynomials. Furthermore, once this is known, we can also demonstrate that $a_1,...,a_m$ are also continuous exponential polynomials. This follows from the fact that the translation operator $\tau_y(f)(x)=f(x+y)$ is continuous, which implies that $a_1,\cdots,a_m$ are continuous functions and then a symmetry argument (interchange $x$ and $y$) will show that they are, indeed, continuous exponential polynomials.

Note that when we say that two distributions are equal, this equality is in distributional sense. Hence, if a distribution $u$ is equal to a continuous function $f$, this means that the distribution is an ordinary function  and it is equal almost everywhere, with respect to Lebesgue measure, to $f$. Thus, when we claim that a distribution is a continuous  exponential polynomial, we just state equality almost everywhere in Lebesgue sense.

If $W$ is a vector subspace of $X_d$ and $f,\tau_y(f)\in W$, then $\Delta_yf=\tau_y(f)-f \in W$. Thus, for the additive group 
$\Gamma=\mathbb{R}^d$, Levi-Civita functional equation also takes the form
\begin{equation} \label{LC1}
\Delta_yf(x)\in W \ \ (y\in\mathbb{R}^d),
\end{equation}
with $W$  a finite dimensional subspace of $X_d$.

In \cite{AS_LCFE} the  functional equation
\[
\sum_{i=0}^mf_i(b_ix+c_iy)= \sum_{i=1}^na_i(y)v_i(x)
\]
with $x,y\in\mathbb{R}^d$ and $b_i,c_i\in\mathbf{GL}_d(\mathbb{C})$, was studied, both in the classical context of continuous complex valued functions and in the  framework of complex valued Schwartz distributions. Concretely, the following results were demonstrated (see \cite[Theorems 3, 6, 9 and Corollary 9]{AS_LCFE}):

\begin{theorem} \label{principal}
Assume that $\{f_k\}_{k=1}^{m}\subset X_d$ and, for all $y\in \mathbb{R}^d$,
\begin{equation} \label{delCP1} 
\sum_{i=1}^m\tau_{c_iy}(f_i) \in W \text{ for all } y\in\mathbb{R}^d
\end{equation}
for an $n$-dimensional subspace $W$ of $\mathcal{D}(\mathbb{R}^d)'$. If all matrices $c_i$ and $c_i-c_j$ (for $i\neq j$) are invertible,
then all $f_k$ are continuous exponential polynomials.
\end{theorem}

\begin{theorem} \label{principal2}
Assume that
\begin{equation} \label{delCP2}  
\sum_{i=1}^mf_i(b_ix+c_iy)= \sum_{k=1}^nu_k(y)v_k(x),
\end{equation}
where $f_i, u_k, v_k\in\mathcal{D}(\mathbb{R}^d)'$ and $b_i,c_i\in {GL}(d,\mathbb{R})$. 
If all matrices $b_i^{-1}c_i-b_j^{-1}c_j$ (for $i\neq j$) are invertible, then $f_k$ is 
a continuous exponential polynomial for $k=1,\ldots,m$.
\end{theorem}

\begin{theorem} \label{GO-t}
Assume that $f_i, a_{\alpha}, b_{\beta}\in  \mathcal{D}(\mathbb{R}^d)'$ for $1\leq i\leq m$, $0\leq |\alpha|\leq r$ and $0\leq |\beta|\leq s$, and equation
 \begin{equation}\label{G-O}
\sum_{i=1}^mf_i(x+c_iy)=A(x,y)+B(y,x),
\end{equation}
 is satisfied with  $A(x,y)=\sum_{|\alpha|\leq r}x^{\alpha} \cdot a_{\alpha}(y)$ and   $B(y,x)=\sum_{|\beta|\leq s}b_{\beta}(x)\cdot y^{\beta}$. Assume, furthermore, that all matrices $c_i$ (for all $i$) and $c_i-c_j$ (for $i\neq j$) are invertible. Then all $f_i$ are (in the distributional sense) ordinary polynomials.
\end{theorem}

Thus, the solution sets of these equations are, typically,  exponential polynomials and, in some particular cases, they reduce to ordinary polynomials. 

In this paper we present several characterizations of ordinary polynomials as the solution sets of certain functional equations. Some of  these equations are important because of their connection with the  characterization problem of distributions in Probability Theory, by using the characteristic functions of random variables and random vectors. In particular,  the equation \eqref{delCP2} , when restricted to ordinary functions $f_i$, $A$ and $B$,  has been recognized as an useful tool in the  characterization problem for Gaussian distributions (see, for example, \cite[Chapter 7]{MaPe}). 

\begin{remark}
Note that the equations considered in Theorems \ref{principal2} and \ref{GO-t} are considered as a single equation in $\mathcal{D}(\mathbb{R}^d\times \mathbb{R}^d)'$. To do this, it is necessary to fix the notation we are using. Concretely, if $f_i, a_k, v_k\in\mathcal{D}(\mathbb{R}^d)'$, and $b_i, c_i$ are nvertible matrices,  the distributions $f_i(b_ix+c_iy), a_k(y)v_k(x)\in \mathcal{D}(\mathbb{R}^d\times \mathbb{R}^d)'$ are defined by
\[
f_i(b_ix+c_iy)\{\varphi_1(x)\varphi_2(y)\}=f_{i}\left\{\frac{1}{|\det(b_i)||\det(c_i)|}\sigma_{b_i^{-1}}(\varphi_1)\ast \sigma_{c_i^{-1}}(\varphi_2)\right\}
\]   
and
\[
a_k(y)v_k(x)\{\varphi_1(x)\varphi_2(y)\}= a_k(y)\{\varphi_2(y)\} v_k(x)\{\varphi_1(x)\}
\]
for arbitrary $\varphi_1,\varphi_2\in \mathcal{D}(\mathbb{R}^d)$,  respectively. Indeed, the vector subspace of $\mathcal{D}(\mathbb{R}^d\times \mathbb{R}^d)$ spanned by the functions of the form  $\varphi_1(x)\varphi_2(y)$, with $\varphi_1,\varphi_2\in \mathcal{D}(\mathbb{R}^d)$ is dense in $\mathcal{D}(\mathbb{R}^d\times \mathbb{R}^d)$. Hence, any distribution $F\in \mathcal{D}(\mathbb{R}^d\times \mathbb{R}^d)'$ is completely determined from its values on these products \cite[p. 51]{V}. 
\end{remark}

\begin{remark}
Note that, if the matrices $\{b_1,\cdots,b_m,c_1,\cdots,c_m\}$ are pairwise commuting, then the statements below are equivalent:
\begin{itemize}
\item  $b_i^{-1}c_i-b_j^{-1}c_j$ is invertible whenever $i\neq j$.
\item $b_ic_j-b_jc_i$ is invertible whenever $i\neq j$
\end{itemize}
This is so because, if $b_i^{-1}c_i-b_j^{-1}c_j$ is invertible, then $b_jb_i(b_i^{-1}c_i-b_j^{-1}c_j)= b_jc_i-b_ic_j$ is also invertible,  and, for the reverse implication we can multiply $b_jc_i-b_ic_j$ by $b_i^{-1}b_j^{-1}$ to get $b_i^{-1}b_j^{-1}( b_jc_i-b_ic_j) = b_i^{-1}c_i-b_j^{-1}c_j$.
\end{remark}

\begin{remark} If $f,g\in \mathcal{S}'(\mathbb{R}^d)$ are tempered distributions and $b,c\in\mathbf{Gl}_d(\mathbb{R})$ are invertible matrices, then  it is easy to prove that $f(bx+cy), f(x)g(y)\in \mathcal{S}'(\mathbb{R}^d\times \mathbb{R}^d)$, and $\tau_{ch}(f) \in  \mathcal{S}'(\mathbb{R}^d)$ for all $h\in\mathbb{R}^d$. Hence, the equations \eqref{delCP1}  and  \eqref{delCP2}  can be studied for tempered distributions. Now, $\mathcal{S}'(\mathbb{R}^s)\subset \mathcal{D}(\mathbb{R}^s)'$ for all $s\in\mathbb{N}$, so that Theorems  \ref{principal} and  \ref{principal2} and the fact that all continuous exponential polynomials are tempered distributions, imply that all results in this section can be  applied also under the additional restriction that all  distributions under consideration are tempered distributions, and  the same holds for all the equations studied along this paper. 
\end{remark}

\section{Polynomials as solution sets of certain Functional equations connected to Probability Theory}
As Theorem \ref{GO-t} shows, there are some particular cases of equation \eqref{delCP2}  which characterize  continuous polynomials in $\mathbb{R}^d$. 
In this section we present a new proof of Theorem \ref{GO-t} which, in contrast to the one included in \cite{AS_LCFE}, is not based on the very technical tools used in \cite{Gu_O}. We also present a new proof of a classical result by Ghurye and Olkin about  characterizations of Gaussian distributions (see Theorem \ref{GOT} below), under certain additional restrictions. 

\subsection{New proof of Theorem \ref{GO-t}}

Let us first study the case $m=1$ of the equation:

\begin{lemma}\label{case1}
Assume that  $c\in\mathbf{Gl}_d(\mathbb{R})$ and 
\[
f(x+cy)=\sum_{|\alpha|\leq r}x^{\alpha} \cdot a_{\alpha}(y)+\sum_{|\beta|\leq s}b_{\beta}(x)\cdot y^{\beta}
\]
for some complex valued distributions $f$, $a_{\alpha},b_{\beta}\in\mathcal{D}(\mathbb{R}^d)'$. Then $f$ is a polynomial. 
\end{lemma}

\begin{proof}
Obviously, this equation is a particular case of  \eqref{delCP2}. Hence $f$ is an exponential polynomial, which means that  $f=\sum_{k=1}^{m} p_{k}(x)e^{\langle \lambda_{k},x\rangle}$ for certain $m$ and certain polynomials $p_{k}$ and complex vectors $\lambda_{k}\in\mathbb{C}^d$. Hence 
\begin{equation}
 \sum_{k=1}^{m} p_{k}(x+cy)e^{\langle \lambda_{k},x+cy\rangle}=\sum_{|\alpha|\leq r}x^{\alpha} \cdot c_{\alpha}(y)+\sum_{|\beta|\leq s}d_{\beta}(x)\cdot y^{\beta}=:\Phi(x,y)
\end{equation}

Given $h_1,h_2\in\mathbb{R}^d$, we have that  $\Delta_{(h_1,0)}^{r+1}\Delta_{(0,h_2)}^{s+1}\Phi(x,y)=0$. On the other hand, 
\begin{eqnarray*}
&\ & \Delta_{(h_1,0)}(p_{k}(x+cy)e^{\langle \lambda_{k},x+cy\rangle}) \\
&=& \ \ \ p_{k}(x+cy+h_1)e^{\langle \lambda_{k},x+cy+h_1\rangle}- p_{k}(x+cy)e^{\langle \lambda_{k},x+cy\rangle}\\
&=&\ \ \  \left(p_{k}(x+cy+h_1)e^{\langle \lambda_{k},h_1\rangle}- p_{k}(x+cy)\right)e^{\langle \lambda_{k},x+cy\rangle}\\
&=&\ \ \  \left(e^{\langle \lambda_{k},h_1\rangle}\tau_{h_1}- 1_d\right)(p_{k})(x+cy)e^{\langle \lambda_{k},x+cy\rangle}
\end{eqnarray*}
and 
 \begin{eqnarray*}
&\ & \Delta_{(0,h_2)}(p_{k}(x+cy)e^{\langle \lambda_{k},x+cy\rangle}) \\
&=& \ \ \ p_{k}(x+cy+ch_2)e^{\langle \lambda_{k},x+cy+ch_2\rangle}- p_{k}(x+cy)e^{\langle \lambda_{k},x+cy\rangle}\\
&=&\ \ \  \left(p_{k}(x+cy+ch_2)e^{\langle \lambda_{k},ch_2\rangle}- p_{k}(x+cy)\right)e^{\langle \lambda_{k},x+cy\rangle}\\
&=&\ \ \  \left(e^{\langle \lambda_{k},ch_2\rangle}\tau_{ch_2}- 1_d\right)(p_{k})(x+cy)e^{\langle \lambda_{k},x+cy\rangle}.
\end{eqnarray*}
Hence 
\begin{eqnarray*}
0 &=& \Delta_{(0,h_2)}^{r+1}\Delta_{(0,h_2)}^{s+1} \Phi(x,y) \\
&=& \sum_{k=1}^{m} \Delta_{(0,h_2)}^{r+1}\Delta_{(0,h_2)}^{s+1} 
\left(p_{k}(x+cy)e^{\langle \lambda_{k},x+cy\rangle}\right) \\
&=& \sum_{k=1}^{m}  
 \left(e^{\langle \lambda_{k},h_1\rangle}\tau_{h_1}- 1_d\right)^{s+1}\left(e^{\langle \lambda_{k},ch_2\rangle}\tau_{ch_2}- 1_d\right)^{r+1}(p_{k})(x+cy)e^{\langle \lambda_{k},x+cy\rangle}
\end{eqnarray*}
Hence
\[
\sum_{k=1}^{m}  
 \left(e^{\langle \lambda_{k},h_1\rangle}\tau_{h_1}- 1_d\right)^{s+1}\left(e^{\langle \lambda_{k},c h_2\rangle}\tau_{ch_2}- 1_d\right)^{r+1}(p_{k})(x)e^{\langle \lambda_{k},x\rangle}=0,
\]
since, for any distribution $u$ and any invertible matrix $c\in\mathbf{Gl}_d(\mathbb{R})$, we have that $u(x+cy)=0$ as an element of $\mathcal{D}(\mathbb{R}^d\times \mathbb{R}^d)'$ if and only if $u=0$ as an element of $ \mathcal{D}(\mathbb{R}^d)'$.  

We can assume that vectors $\lambda_{k}$, $k=1,\cdots,m$ are pairwise distinct and all $p_{k}$ are different from  $0$. 
Then   \cite[Lemma 4.3]{Sz1} (see also \cite[Theorem 5.10]{Sz}) implies that 
\[
 \left(e^{\langle \lambda_{k},h_1\rangle}\tau_{h_1}- 1_d\right)^{s+1}\left(e^{\langle \lambda_{k},c_1h_2\rangle}\tau_{ch_2}- 1_d\right)^{r+1}(p_{k})(x)=0,\ \ k=1,\cdots,m.
\]
for all $h_1,h_2\in\mathbb{R}^d$. In particular, taking $h_2=c^{-1}h_1$, we have that 
\[
 \left(e^{\langle \lambda_{k},h_1\rangle}\tau_{h_1}- 1_d\right)^{s+r+2}(p_{k})(x)=0,\ \ k=1,\cdots,m,
\]
for all $h_1\in\mathbb{R}^d$. Consider the polynomial $p_k$ and the variety it generates: $V_k=\tau(p_k)=\mathbf{span}\{p_k(x+\alpha):\alpha\in\mathbb{R}^d\}$, which is a translation invariant space of finite dimension. Then for any operator $L:V_k \to V_k$ which commutes with translations  we have that  $L=0$ if and only if $L(p_k)=0$.  Moreover,  $\tau_h:V_k\to V_k$ defines an automorphism for every $h\in\mathbb{R}^d$, and there exists an integral number $M>0$ such that  $\Delta_h^Mp_k=0$ for all $h\in\mathbb{R}^d$, since $p_k$ is a polynomial. But $\Delta_h^M=(\tau_h-1_d)^M$ implies that $(z-1)^M$ is a multiple of the minimal polynomial associated to the operator $\tau_h:V_k\to V_k$. This implies that the only eigenvalue of $\tau_h$ is $1$ and the minimal polynomial of $\tau_h$ is of the form $(z-1)^{m(h)}$ for some $m(h)>1$, for all $h\neq 0$ (since $\tau_h$  is different from the identity).  Hence, $\left(e^{\langle \lambda_{k},h_1\rangle}z- 1\right)^{s+r+2}$ must be a multiple of $(z-1)$ for every $h_1$. In particular, $e^{\langle \lambda_{k},h_1\rangle}=1$, for all $h_1$, which implies that $\lambda_{k}=\mathbf{0}$. Hence $k=1$, $\lambda_{1}=\mathbf{0}$, and  $f$ is a polynomial.
\end{proof}

\begin{proof}[Proof of Theorem \ref{GO-t}]  Lemma above proves the case $m=1$. 
Assume $m>1$ and apply the operator $\Delta_{(h_1,c_1^{-1}h_1)}$ to both sides of the equation. This transforms it into a similar one, with $m-1$ summands,
\begin{equation}
\sum_{i=2}^{m} g_i(x+c_iy) = \sum_{|\alpha|\leq r}x^{\alpha} \cdot a_{\alpha}^*(y)+\sum_{|\beta|\leq s}b_{\beta}^*(x)\cdot y^{\beta}
\end{equation}
with $g_i= \Delta_{(I_d-c_ic_1^{-1})h_1}(f_i)\in\mathcal{D}(\mathbb{R}^d)'$ for $i=2,\cdots,m$.  Thus, we can apply the operator 
$\Delta_{(h_2,c_2^{-1}h_2)}$ to both sides of the new equation to reduce again by one the number of summands, 
\begin{equation}
\sum_{i=3}^{m}  \Delta_{(I_d-c_ic_2^{-1})h_2} \Delta_{(I_d-c_ic_1^{-1})h_1}(f_i)(x+c_iy) = \sum_{|\alpha|\leq r}x^{\alpha} \cdot a_{\alpha}^{**}(y)+\sum_{|\beta|\leq s}b_{\beta}^{**}(x)\cdot y^{\beta},
\end{equation}
and, iterating the process, we get an equation with just one summand in the left hand side member:
\begin{eqnarray*}
&\ & \Delta_{(I_d-c_mc_{m-1}^{-1})h_{m-1}} \cdots \Delta_{(I_d-c_mc_2^{-1})h_2} \Delta_{(I_d-c_mc_1^{-1})h_1}(f_m)(x+c_my) \\
&\ & \ \ = \sum_{|\alpha|\leq r}x^{\alpha} \cdot a_{\alpha}^{***}(y)+\sum_{|\beta|\leq s}b_{\beta}^{***}(x)\cdot y^{\beta}.
\end{eqnarray*}
Applying case $m=1$ we conclude that $$\Delta_{(I_d-c_mc_{m-1}^{-1})h_{m-1}} \cdots \Delta_{(I_d-c_mc_2^{-1})h_2} \Delta_{(I_d-c_mc_1^{-1})h_1}(f_m)$$ is a polynomial for all $h_1,\cdots,h_{m-1}\in\mathbb{R}^d$. In particular, $f_m$  satisfies the distributional version of Fr\'{e}chet's functional equation $\Delta_h^{N+1}f_m=0$, for all $h$ and for certain $N$. Hence $f_m$ is a polynomial (see  \cite{AA_AM},  \cite{A_NFAO}, \cite{AK_CJM} or  \cite{AS_AM}). A simple permutation in the summation process, before applying the very same arguments above, produces the very same result for $f_i$ for all $i$. 
\end{proof}

A similar result to Theorem \ref{GO-t} was already proved by Ghurye and Olkin for continuous functions. Indeed, for $d>1$, Theorem \ref{GO-t} generalizes to distributional setting, with a simpler proof, the particular case we get from  \cite[Lemma 3]{Gu_O} when we impose the additional condition $\det(c_i-c_j)\neq 0$ for $i\neq j$. For $d=1$, see  \cite[Lemma 4]{Gu_O}.

\subsection{New proof of Ghurye and Olkin Theorem} 

In this section we demonstrate the following result, which serves for a new proof of a well known result by Ghurye and Olkin:
\begin{theorem} \label{FTI}  Assume that   $P,Q,f_i\in \mathcal{D}(\mathbb{R}^d)'$, $i=1,\cdots,m$, and  $b_i,c_i\in \mathbf{Gl}_d(\mathbb{R})$ are such that  $b_i^{-1}c_i-b_j^{-1}c_j$ is invertible whenever $i\neq j$. If
\begin{equation}\label{Variation}
\sum_{i=1}^{m}f_i(b_ix+c_iy)= P(x)\cdot 1(y) +   1(x) \cdot  Q(y),
\end{equation} 
where  the equality is understood in the sense of $\mathcal{D}(\mathbb{R}^d\times \mathbb{R}^d)'$. Then $P(x)$ and   $Q(y)$ are, in distributional sense, continuous polynomials of degree $\leq m$. Furthermore, if $Q=0$ in \eqref{Variation} then the degree of $P$ is $m-1$ at most. Analogously, if $P=0$, in \eqref{Variation} then the degree of $Q$ is $m-1$ at most.
\end{theorem} 
Indeed, if we impose $P(x)=\sum_{i=1}^m f_i(b_ix)$ and   $Q(y)=\sum_{i=1}^m f_i(c_iy)$, the equation \eqref{Variation} is transformed into 
\begin{equation}\label{Ski}
\sum_{i=1}^{m}f_i(b_ix+c_iy)= \sum_{i=1}^m f_i(b_ix)\cdot 1(y) +   1(x) \cdot  \sum_{i=1}^m f_i(c_iy),
\end{equation}
which is a distributional version of the equation you get taking logarithms at both sides of Skitovich-Darmois functional equation:
\begin{equation}\label{Ski-Dar}
\prod_{i=1}^m \widehat{\mu_i}(b_ix+c_iy) = \prod_{i=1}^m \widehat{\mu_i}(b_ix)  \prod_{i=1}^m  \widehat{\mu_i}(c_iy).
\end{equation}
Here, $ \widehat{\mu_i}$ represents the characteristic function of a probability distribution $\mu_i$. If $f_i=-\log  \widehat{\mu_i}$ for all $i$, then we get the equation \eqref{Ski} for ordinary functions. This equation is connected to the characterization problem of Gaussian distributions. Concretely, its study leads to a proof of the following result: 

\begin{theorem}[Ghurye-Olkin \cite{Gu_O}, \cite{KLR}] \label{GOT} Assume that $X_i$, $i=1,\cdots,m$ are independent $d$-dimensional random vectors such that the linear forms $L_1= b_1^tX_1+...+b_m^tX_m$ and $L_2= c_1^tX_1+...+c_m^tX_m$  are independent, with $b_i, c_i\in\mathbf{Gl}_d(\mathbb{R})$ for $i=1,\cdots,m$. Then $X_i$ is Gaussian for all $i$.
\end{theorem}

Previous to demonstrating Theorem \ref{FTI}, and for the sake of completeness, we include here the  arguments which justify that equation \eqref{Variation} is connected to Ghurye-Olkin's Theorem. If $\mu_i$ is the distribution function of the random vector $X_i$ and $\widehat{\mu_i}(x)=E(e^{i \langle x,X_i\rangle })$ is its characteristic function, the independence of $L_1$ and $L_2$ implies that 
\[
E(e^{i(\langle x, L_1\rangle +\langle y,L_2\rangle)})=E(e^{i\langle x,L_1\rangle }e^{i\langle y,L_2\rangle}) = 
E(e^{i\langle x, L_1\rangle})E(e^{i\langle y,L_2\rangle})
\]
Now,  
\begin{eqnarray*}
E(e^{i(\langle x, L_1\rangle +\langle y,L_2\rangle)}) &=& E(e^{i(\langle x,b_1^tX_1+...+b_m^tX_m\rangle+\langle y, c_1^tX_1+...+c_m^tX_m\rangle )}) \\
&=& E(e^{i(\langle b_1x+c_1y,X_1\rangle +...+\langle b_mx+c_my,X_m\rangle)}) \\
&=& E(e^{i\langle b_1x+c_1y,X_1\rangle} \cdots e^{i \langle b_mx+c_my,X_m\rangle}) \\
&=& E(e^{i\langle b_1x+c_1y,X_1\rangle}) \cdots E(e^{i \langle b_mx+c_my,X_m\rangle}) \\
& =& \prod_{i=1}^m \widehat{\mu_i}(b_ix+c_iy),
\end{eqnarray*}
since the $X_i$'s are independent. With analogous computations, we get 
\[
E(e^{i\langle x, L_1\rangle})=  \prod_{i=1}^m \widehat{\mu_i}(b_ix)
\]
and
\[
E(e^{i\langle x, L_2\rangle})=  \prod_{i=1}^m  \widehat{\mu_i}(c_iy).
\]
Hence the characteristic functions $\widehat{\mu_i}$ satisfy the Skitovich-Darmois functional equation \eqref{Ski}. Now, the assumption that they are characteristic functions, and that matrices $b_i,c_i$ are invertible, imply that, for each $i$, $\widehat{\mu_i}$  doesn't vanish anywhere (this was proved, e.g., in \cite[Lemma 1]{Gu_O}). Hence we can take logarithms at both sides of the equation to get the  Levi-Civita type functional equation 
\begin{equation}\label{Ski-ordinary}
\sum_{i=1}^{m}f_i(b_ix+c_iy)= \sum_{i=1}^m f_i(b_ix) +   \sum_{i=1}^m f_i(c_iy),
\end{equation}
where $f_i=-\log \widehat{\mu_i}$. Now, Theorem \ref{FTI} implies that, under the additional assumption that $b_i^{-1}c_i-b_j^{-1}c_j$ is invertible for all $i\neq j$, the functions $P(x)= \sum_{i=1}^m f_i(b_ix)$ and $Q(y)=   \sum_{i=1}^m f_i(c_iy)$ are polynomials. Hence the characteristic function of, for example, $L_1$, is of the form $e^{P(x)}$ with $P(x)$ a polynomial. Then Marcinkiewicz's theorem \cite[Lemma 1.4.2]{KLR} implies that $P(x)$ is a quadratic polynomial and $L_1$ is Gaussian. Finally, Cram\'{e}r's theorem for $\mathbb{R}^d$ claims that,  in this case, all the random vectors $X_i$ must be Gaussian (see \cite[p. 112]{Cra}). This proves Theorem \ref{GOT} under the additional assumption that $\det(b_i^{-1}c_i-b_j^{-1}c_j)\neq 0$ for all $i\neq j$.

The equation \eqref{Ski-ordinary} and its applications to the characterization problems of probability distributions, has been studied with great detail, by Feld'man \cite{Feld}, for ordinary functions defined on quite general commutative groups. May be our contribution here is that we formulate the equation, for the very first time, in distributional setting. 

\begin{proof}[Proof of Theorem \ref{FTI}]
There is no loss of generality if we assume $b_i=I_d$ for $i=1,\cdots, m$. Hence our equation takes the form
\begin{equation}\label{Sn2}
\sum_{i=1}^{m}f_i(x+c_iy)= P(x)\cdot 1(y) +   1(x) \cdot Q(y),
\end{equation}
with  $c_i$ invertible for all $i$,  and $c_i-c_j$ invertible whenever $i\neq j$. Take $h_1\in\mathbb{R}^d$. and apply the operator $\Delta_{(h_1,-c_1^{-1}h_1)}$ to both sides of the equation. Then, in the left hand side we get 
\begin{equation}
\Delta_{(h_1,-c_1^{-1}h_1)}\left[ \sum_{i=1}^{m} f_i(x+c_iy)\right]  = \sum_{i=2}^{m} g_i(x+c_iy),
\end{equation}
with $g_i= \Delta_{(I_d-c_ic_1^{-1})h_1}(f_i)\in\mathcal{D}(\mathbb{R}^d)'$ for $i=2,\cdots,m$.  On the other hand, in the right hand side of the equation we get 
\[
\Delta_{(h_1,-c_1^{-1}h_1)}\left(P(x)\cdot 1(y)\right)  = \tau_{-c_1^{-1}h_1}(1)(y)\cdot \tau_{h_1}(P)(x) - 1(y)\cdot P(x) = \Delta_{h_1}P(x) \cdot 1(y).
\] 
and 
\[
\Delta_{(h_1,-c_1^{-1}h_1)}\left(1(x)\cdot Q(y)\right)  = \tau_{-c_1^{-1}h_1}(Q)(y)\cdot \tau_{h_1}(1)(x) - Q(y)\cdot 1(x) = 1(x)\cdot \Delta_{-c_1^{-1}h_1}Q(y).
\] 
Hence 
\[
 \sum_{i=2}^{m} g_i(x+c_iy) = \Delta_{h_1}P(x)\cdot 1(y) + 1(x)\cdot \Delta_{-c_1^{-1}h_1}Q(y)
\]
A simple iteration of the argument $m$ times leads to the equation:
\begin{equation} \label{PQsuma}
0= (\Delta_{h_{m}}\cdots  \Delta_{h_2}\Delta_{h_1}P(x)) \cdot 1(y) + 1(x)\cdot  
(\Delta_{-c_m^{-1}h_m}\cdots  \Delta_{-c_2^{-1}h_2}\Delta_{-c_1^{-1}h_1}Q(y)) ,
\end{equation}
which implies that  $\Delta_{h_{m}}\cdots  \Delta_{h_2}\Delta_{h_1}P(x)$ and $\Delta_{-c_m^{-1}h_m}\cdots  \Delta_{-c_2^{-1}h_2}\Delta_{-c_1^{-1}h_1}Q(y)$ are both constant functions. In particular, taking $h_{m+1}\in\mathbb{R}^d$ and applying the operator $\Delta_{(h_{m+1},0)}$ to both sides of the equation, we get
$$
\Delta_{h_{m+1}}\Delta_{h_{m}}\cdots  \Delta_{h_2}\Delta_{h_1}P(x) \cdot 1(y) =0
$$
and, analogously, if we apply $\Delta_{(0,h_{m+1})}$ to both sides of the equation, we get
$$
1(x)\cdot  \Delta_{h_{m+1}}\Delta_{-c_m^{-1}h_m}\cdots  \Delta_{-c_2^{-1}h_2}\Delta_{-c_1^{-1}h_1}Q(y)=0
$$
Thus 
$$
\Delta_{h_{m+1}}\Delta_{h_{m}}\cdots  \Delta_{h_2}\Delta_{h_1}P(x)=  \Delta_{h_{m+1}}\Delta_{-c_m^{-1}h_m}\cdots  \Delta_{-c_2^{-1}h_2}\Delta_{-c_1^{-1}h_1}Q(y)=0
$$
for all $h_1,\cdots, h_{m+1}$ in $\mathbb{R}^d$ (with equality in the sense of $\mathcal{D}(\mathbb{R}^d)'$), and the result follows from the corresponding Fr\'{e}chet's type theorem for distributions, which is known (see again  \cite{AA_AM},  \cite{A_NFAO}, \cite{AK_CJM} or  \cite{AS_AM}).
Note that, if we assume $Q=0$ in the hypotheses of the theorem, then equation  \eqref{PQsuma} takes the form 
\[
0= (\Delta_{h_{m}}\cdots  \Delta_{h_2}\Delta_{h_1}P(x)) \cdot 1(y), 
\]
which leads to $\Delta_{h_{m}}\cdots  \Delta_{h_2}\Delta_{h_1}P(x)=0$ and henceforth, in that case, $P$ is a polynomial of degree $\leq m-1$. Of course, an analogous argument proves that, if $P=0$ in \eqref{Variation}, then $Q$ is a polynomial of degree $\leq m-1$.
\end{proof}

\begin{remark}
Note that the equation 
\begin{equation*}
\sum_{i=1}^mf_i(x+c_iy)=\sum_{|\alpha|\leq r}x^{\alpha} \cdot a_{\alpha}(y)+\sum_{|\beta|\leq s}b_{\beta}(x)\cdot y^{\beta}
\end{equation*}
generalizes \eqref{Variation} (It coincides, when $r=s=0$). Hence Theorem \ref{FTI} follows, with a completely different proof, as a direct corollary of Theorem \ref{GO-t}, since, if $\Phi(x,y)= P(x)\cdot 1(y) +1(x)\cdot Q(x)$ is a polynomial in $\mathbb{R}^d\times \mathbb{R}^d$, then both $P$ and $Q$ are polynomials in $\mathbb{R}^d$. \end{remark}

\section{A distributional version of Fr\'{e}chet's and Kakutami-Nagumo-Walsh's  theorems}
Given $f\in\mathcal{D}(\mathbb{R}^d)'$, we can think of $F= \Delta_y^{m}f(x)$ as a distribution in $\mathbb{R}^d\times\mathbb{R}^d$. To distinguish this from the standard interpretation of the operator $\Delta_y^m$, we introduce a new notation:  
 \[
 \widetilde{\Delta}_y^{m}f(x)= f(x)\cdot 1(y)+\sum_{i=1}^m\binom{m}{i}(-1)^{m-i}f(I_dx+(iI_d)y),
 \]  
and then we can prove the following  generalization of Fr\'{e}chet's theorem:
 
 \begin{theorem}[Fr\'{e}chet's type theorem] \label{FTT1}
 Let $f\in\mathcal{D}(\mathbb{R}^d)'$ be such that     
 \begin{equation}\label{FG}
 \widetilde{\Delta}_y^{m}f(x)= f(x)\cdot 1(y)+\sum_{i=1}^m\binom{m}{i}(-1)^{m-i}f(I_dx+(iI_d)y)=0,
 \end{equation}
where the equality is understood in the sense of $\mathcal{D}(\mathbb{R}^d\times\mathbb{R}^d)'$. Then $f$ is equal almost everywhere to a continuous polynomial function in 
$\mathbb{R}^d$ of total degree $\leq m-1$.  \end{theorem}

\begin{proof} Use Theorem \ref{FTI} with $P=f$, $Q=0$, $f_i= \binom{m}{i}(-1)^{m+1-i}f$,  $b_i=I_d$,  and  $c_i=iI_d$, $i=1,\cdots,m$, which obviously satisfy that  $b_i^{-1}c_i-b_j^{-1}c_j=(i-j)I_d$ is invertible whenever $i\neq j$.
\end{proof}

We  end this section with a commentary about harmonic polynomials  of degree at most $N$ (i.e, real and imaginary parts of complex polynomials $p(z)=a_0+a_1z\cdots+a_Nz^N$).  S. Kakutani and M. Nagumo \cite{kn} and J. L.  Walsh \cite{wal} introduced, in the 1930's, the functional equation 
\begin{equation}\label{KNW}
\frac{1}{N}\sum_{k=0}^{N-1} f(z+w^kh)=0 \text{, for all } z,h\in\mathbb{C},
\end{equation}
where $w$ is any primitive $N$-th root of $1$. This equation was extensively studied by S. Haruki \cite{h1,h2,h3} in the 1970's  and   1980's. Its continuous solutions $f:\mathbb{C}\to\mathbb{R}$  are  harmonic polynomials of degree at most $N$ and, if we restrict our attention to  solutions $f:\mathbb{C}\to\mathbb{C}$ which are holomorphic, the equation characterizes the complex polynomials of degree at most $N$. Now we can demonstrate the following
\begin{theorem}
Assume that  $f\in\mathcal{D}(\mathbb{R}^2)'$ is a solution of the equation \eqref{KNW}, which we formulate  as an identity in $\mathcal{D}(\mathbb{R}^2\times \mathbb{R}^2)'$. Then $f$ is equal, in distributional sense, to an harmonic polynomial of degree at most $N$.
\end{theorem}
\begin{proof} Under these hypotheses,  Theorem \ref{FTI} can be applied to $f$, since  the linear operator $L_a:\mathbb{C}\to \mathbb{C}$ given by  $L_a(z)= az$ is invertible as soon as $a\in \mathbb{C}\setminus\{0\}$ and, if $w$ is a primitive $N$-th root of $1$ and  $0\leq k,l\leq N-1$, $k\neq l$, then $w^k-w^l\neq 0$. Hence $f$ is a polynomial in $\mathbb{R}^2$ which satisfies the Kakutani-Nagumo-Walsh equation \eqref{KNW} and, henceforth, is an harmonic polynomial.
\end{proof}

\section{Another distributional characterization of polynomials}

In this section we introduce another variation of the distributional version of Fr\'{e}chet's Theorem, different in spirit of   Theorem \ref {FTT1} above, which represents, up to our knowledge, a new characterization of ordinary polynomials  which is, furthermore,  stronger than the original Fr\'{e}chet's theorem. 
 
Let $f\in \mathcal{D}(\mathbb{R}^d)'$ be a complex valued distribution defined on $\mathbb{R}^d$, and let  $q(z)=a_0+a_1z+\cdots+a_nz^n$ be a polynomial of one variable. We define, for each $y\in \mathbb{R}^d$,  the new distribution 
$$q(\tau_y)(f)=\sum_{k=0}^na_k (\tau_y)^k(f)$$
which, in case of $f$ being an ordinary function, can be written as
$$q(\tau_y)(f)(x)=\sum_{k=0}^na_k f(x+ky).$$

The following result is an improvement of \cite[Theorem 2.2.]{A_JJA}:

\begin{theorem} \label{new_char_pol} Assume $d>1$. Let $f$ be a complex valued distribution defined on $\mathbb{R}^d$. Assume that
$q(\tau_y)(f)=0 $  for all  $y\in\mathbb{R}^d $  with  $\|y\| = \delta$, for certain $\delta>0$ and certain polynomial $q(z)=a_0+a_1z+\cdots a_nz^n$, with $a_n\neq 0$.  
Then $f$ is, in distributional sense, an ordinary polynomial. In particular, if $f$ is a continuous function, then it is an ordinary polynomial. 
\end{theorem}

Let us first prove two technical Lemmas:
\begin{lemma} Consider on $\mathbb{R}^d$ the Euclidean norm, the sphere $S_d(\delta)=\{x\in\mathbb{R}^d: \|x\|=\delta\}$, and the ball 
$B_d(\delta)=\{x\in\mathbb{R}^d: \|x\|\leq \delta\}$. If $d>1$, then 
$$B_d(2\delta)= S_d(\delta)-S_d(\delta)$$ 
\end{lemma}
\begin{proof} The inclusion $S_d(\delta)-S_d(\delta)\subseteq B_d(2\delta)$ follows directly from triangle inequality, for any norm we consider on $\mathbb{R}^d$. Let us demonstrate the other inclusion. Take $x\in B_d(2\delta)$. If $\|x\|=2\delta$ then $x=\frac{x}{2}- \frac{-x}{2}$ and we are done. If $\|x\|<2\delta$, take $\Gamma$ any plane passing through the origin of coordinates and containing $x$ as an element (this can be done because $d>1$). The inequality $\|x\|<\delta+\delta$, and the fact that we are considering the Euclidean norm, imply that in this plane there exists a triangle one of whose sides is $x$ and the other two sides have both length $\delta$. If $O,P,Q $ are the vertices of such triangle, with $x=\widearrow{PQ}$, then $x=\widearrow{OQ}-\widearrow{OP}$ and $\|\widearrow{OP}\|=\|\widearrow{OQ}\|=\delta$
\end{proof}

\begin{lemma} \label{K} Assume $d>1$. Then  for any $\delta>0$ there exist points $\{h_1,\cdots,h_s\}\subset S_d(\delta)$ in the sphere of $\mathbb{R}^d$ of radius $\delta$, which span a dense additive subgroup of $\mathbb{R}^d$.
\end{lemma}
\begin{proof}
It follows directly from a well known theorem by Kronecker \cite[Theorem 442, page 382]{HW} (see also \cite{A_popo}) that every open neighborhood of $\mathbf{0}=(0,0,\cdots,0)\in\mathbb{R}^d$ contains a finite set of vectors $\{y_1,\cdots,y_t\}$ which span a dense subgroup of $\mathbb{R}^d$.  In particular, there exist  $\{y_1,\cdots,y_t\}\subseteq B_d(2\delta)$ such that 
$y_1\mathbb{Z}+\cdots+y_t\mathbb{Z}$ is dense in $\mathbb{R}^d$. On the other hand, $d>1$ implies that $B_d(2\delta)= S_d(\delta)-S_d(\delta)$, so that we can take $h_{1}, h_{2},\cdots, h_{2t} \in S_d(\delta)$ such that $y_i=h_{2i}-h_{2i-1}$,  $i=1,\cdots, t$. Then $h_1\mathbb{Z}+\cdots+h_{2t}\mathbb{Z}$ is dense in $\mathbb{R}^d$. 
\end{proof}

\begin{proof}[Proof of Theorem \ref{new_char_pol}]  Let $\{h_1,\cdots,h_s\}\subset S_d(\delta)$ be such that they span a dense subgroup of $\mathbb{R}^d$ (these vectors exist thanks to Lemma \ref{K}). Then $q(\tau_{h_i})(f)=0$ implies that  
$$\dim\mathbf{span}\{f,\tau_{h_i}(f),\cdots, (\tau_{h_i})^n(f)\}\leq n,\text{ for } i=1,\cdots s,$$
and we can apply Theorem 2.1 from \cite{A_popo} to claim that $f$ is, in distributional sense, an exponential polynomial. 

Thus, $f=\sum_{i=1}^r p_i(x)e^{\langle \lambda_i,x\rangle}$ (with equality in distributional sense) for certain complex vectors $\lambda_i\in\mathbb{C}^d$ and certain polynomials $p_i$ (we assume $\lambda_i\neq\lambda_j$ for $i\neq j$). Let us now demonstrate that $r=1$ and $\lambda_1=\mathbf{0}$.  

By hypothesis, $q(\tau_y)(f)=0$ for all $y$ with $\|y\| = \delta$. This means that, for all $y\in S_d(\delta)$,  
\begin{eqnarray*}
0 &=& q(\tau_y)(\sum_{i=1}^r p_i(x)e^{\langle \lambda_i,x\rangle}) = \sum_{i=1}^r q(\tau_y)(p_i(x)e^{\langle \lambda_i,x\rangle}) \\
&=& \sum_{i=1}^r\left(\sum_{k=0}^na_k(\tau_{y})^k(p_i(x)e^{\langle \lambda_i,x\rangle})\right) \\
&=& \sum_{i=1}^r\left(\sum_{k=0}^na_k(e^{\langle \lambda_i,y\rangle})^k(\tau_{y})^k(p_i)(x)\right)e^{\langle \lambda_i,x\rangle} \\
&=& \sum_{i=1}^rQ_{i,y}(\tau_y)(p_i)(x)e^{\langle \lambda_i,x\rangle}, 
\end{eqnarray*}
with $$Q_{i,y}(z)= \sum_{k=0}^na_k(e^{\langle \lambda_i,y\rangle})^kz^k, \text{  } i=1,\cdots,r.$$
It follows that, for all $y\in S_d(\delta)$, $Q_{i,y}(\tau_y)(p_i)$ vanishes identically (see \cite[Lemma 4.3 ]{Sz1}). 

Let us assume that $\lambda_i\neq \mathbf{0}$ and consider the polynomial $p_i$ and the variety it generates: $V_i=\tau(p_i)=\mathbf{span}\{p_i(x+\alpha):\alpha\in\mathbb{R}^d\}$, which is a translation invariant space of finite dimension. The arguments used in the proof of Lemma \ref{case1} for the study of the translation operator $\tau_y: V_i\to V_i$ tell us that the minimal polynomial of this operator is of the form $(z-1)^{m(y)}$ for some $m(y)>1$. Now, $Q_{i,y}(\tau_y)(p_i)=0$ means that $Q_{i,y}((\tau_y)_{|V_i})=0$, which 
implies that $Q_{i,y}(z)$ is necessarily a multiple of  $(z-1)^{m(y)}$. In particular, 
\begin{equation} \label{V1}
\sum_{k=0}^na_k(e^{\langle \lambda_i,y\rangle})^k = Q_{i,y}(1)=0 \text{ for all } y\in S_d(\delta).
\end{equation}
Choose $y_0,\cdots, y_n\in S_d(\delta)$ such that $\rho_t\neq \rho_l$ for all $t\neq l$, where $\rho_t=e^{\langle \lambda_i,y_t\rangle}$ (this can be done because $d>1$ and $\lambda_i\neq \mathbf{0}$).  Then \eqref{V1}, used with these values, leads to the linear system of  equations 
\begin{equation}\label{Po}
\left[\begin{array}{cccccc}
1 & \rho_0 &  \cdots & (\rho_0)^n \\
1 & \rho_1 &  \cdots & (\rho_1)^n \\
\vdots & \vdots & \ \ddots & \vdots \\
1 & \rho_n &  \cdots & (\rho_n)^n \\
\end{array} \right]  
\left[\begin{array}{cccccc}
a_0 \\
a_1 \\
\vdots \\
a_n \\
\end{array} \right] 
= \left[\begin{array}{cccccc}
0 \\
0 \\
\vdots \\
0 \\
\end{array} \right] ,
\end{equation}
which imply $a_0=a_1=\cdots =a_n=0$ since Vandermonde's determinant is different from zero because all $\rho_i$'s are different. This of course contradicts $a_n\neq 0$. Thus $\lambda_i=\mathbf{0}$ necessarily, which ends the proof.  
\end{proof}

\begin{remark} Note that $d>1$ is necessary in Theorem \ref{new_char_pol}, since $S_1(\delta)=\{-\delta, \delta\}$ and, for example, if we set $q(z)=z-1$  and $f(x)=\cos(\frac{2\pi i x}{\delta})$, then $q(\tau_{\pm\delta})(f)(x)= (\Delta_{\pm\delta} f)(x)=0$, but $f$ is not a polynomial. 

On the other hand, for $d=1$, if $f\in\mathcal{D}(\mathbb{R})'$ and  $q(z)=a_0+a_1z+\cdots +a_nz^n$ is a polynomial with $a_n\neq 0$ such that $q(\tau_{y})(f)=0$ for all $y\in U$, for a certain open set $U\subseteq\mathbb{R}$.  Then $U$ contains two elements $h_1,h_2$ such that $h_1/h_2\not \in\mathbb{Q}$, and the very same arguments used for the proof of Theorem \ref{new_char_pol} guarantee that $f$ is (in distributional sense) an ordinary polynomial, since $h_1/h_2\not\in\mathbb{Q}$ guarantees that $\{h_1,h_2\}$ 
span a dense subgroup of $\mathbb{R}$, and Theorem 2.1 from \cite{A_popo} can be used also with $d=1$. Furthermore, for the last part of the proof, we will just need to find, given $n\in\mathbb{N}$ and $\lambda_i$ any non-zero complex number, a set of points  $\{y_1,\cdots,y_n\}\subset U$  such that  $e^{\lambda_iy_t} \neq e^{\lambda_iy_l}$ for all $t\neq l$, and this can (obviously)  be done for any open set $U\subseteq\mathbb{R}$. 
\end{remark}

\end{document}